\documentclass[12pt,leqno]{article}  
\usepackage{amsmath,amssymb,amsthm,bbm}
\usepackage{array,mathdots}
\usepackage{titlesec} 
\usepackage[all,cmtip]{xy}
\usepackage{blkarray,arydshln} 

\titlespacing{\section}{0cm}{3.5pc}{1.5pc}

\makeatletter
\def\@citex[#1]#2{\if@filesw\immediate\write\@auxout{\string\citation{#2}}\fi
  \def\@citea{}\@cite{\@for\@citeb:=#2\do
    {\@citea\def\@citea{\@citesep}\@ifundefined
       {b@\@citeb}{{\bf ?}\@warning
       {Citation `\@citeb' on page \thepage \space undefined}}%
{\csname b@\@citeb\endcsname}}}{#1}}
\def\@citesep{; }
\makeatother

\newtheoremstyle{Kang}{}{}{\itshape}{}{\bf}{}{.5em}{}
\theoremstyle{Kang}
\newtheorem{theorem}{Theorem}[section]

\newtheorem{lemma}[theorem]{Lemma}

\newtheoremstyle{Kremark}{}{}{}{}{\bf}{}{.5em}{}
\theoremstyle{Kremark}
\newtheorem*{remark}{Remark.}
\newtheorem{defn}[theorem]{Definition}
\newtheorem{example}[theorem]{Example}
\newtheorem{other}{}
\newenvironment{idef}[1]{\renewcommand{\theother}{#1}\begin{other}}{\end{other}}

\allowdisplaybreaks[1]  

\def\fn#1{\operatorname{#1}} 
\def\bm#1{\mathbbm{#1}}
\def\c#1{\mathcal{#1}}

\title{Noether's Problem for Some Semidirect Products}

\author{Ming-chang Kang$^{(1)}$ and Jian Zhou$^{(2)}$ \\[3mm]
\begin{minipage}{16cm} \begin{description} \itemsep=-1pt
\item[] $^{(1)}$Department of Mathematics, National Taiwan University, Taipei\\ E-mail: kang@math.ntu.edu.tw
\item[] $^{(2)}$School of Mathematical Sciences, Peking University, Beijing\\ E-mail: zhjn@math.pku.edu.cn
\end{description} \end{minipage}}
\date{}

\begin{document}

\maketitle

\footnote{\textit{\!\!\! $2010$ Mathematics Subject
Classification}. 14E08, 12F10, 13A50, 11R29.}
\footnote{\textit{\!\!\! Keywords and phrases}. Noether's problem, rationality problem, algebraic tori, class groups.}

\begin{abstract}
{\noindent\bf Abstract.}
Let $k$ be a field, $G$ be a finite group, $k(x(g):g\in G)$ be the rational function field with the variables $x(g)$ where $g\in G$.
The group $G$ acts on $k(x(g):g\in G)$ by $k$-automorphisms where $h\cdot x(g)=x(hg)$ for all $h,g\in G$. Let $k(G)$ be the fixed field defined by $k(G):=k(x(g):g\in G)^G=\{f\in k(x(g):g\in G): h\cdot f=f$ for all $h\in G\}$. Noether's problem asks whether the fixed field $k(G)$ is rational (= purely transcendental) over $k$.
Let $m$ and $n$ be positive integers and assume that there is an integer $t$ such that $t\in (\bm{Z}/m\bm{Z})^\times$ is of order $n$. Define a group $G_{m,n}:=\langle\sigma,\tau:\sigma^m=\tau^n=1,\tau^{-1}\sigma\tau=\sigma^t\rangle$ $\simeq C_m \rtimes C_n$. We will find a sufficient condition to guarantee that $k(G)$ is rational over $k$. As a result, it is shown that, for any positive integer $n$, the set $S:=\{p: p$ is a prime number such that $\bm{C}(G_{p,n})$ is rational over $\bm{C} \}$ is of positive Dirichlet density; in particular, $S$ is an infinite set.
\end{abstract}

\section{Introduction}

Let $k$ be any field, $G$ be a finite group and $G\to GL(V)$ be a faithful linear representation of $G$
where $V$ is a finite-dimensional vector space over $k$. Then $G$ acts naturally on the function field $k(V)$ by $k$-automorphisms. Noether's problem asks whether the fixed field $k(V)^G$ is rational (= purely transcendental) over $k$. In particular, when $V=V_{\rm reg}$ is the regular representation, we will write $k(G):=k(V_{\rm reg})^G$. Explicitly, $k(V_{\rm reg}):=k(x(g):g\in G)$ is the rational function field in the variables $x(g)$ (where $g \in G$), $G$ acts on $k(V_{\rm reg})$ by $k$-automorphisms defined by $h\cdot x(g)=x(hg)$ for any $h, g\in G$, and $k(G)=k(V_{\rm reg})^G= \{f\in k(V_{\rm reg}):h\cdot f=f$ for all $h\in G\}$. Note that Noether's problem is a special case of the famous L\"uroth problem.

When the group $G$ is abelian and the field $k$ contains enough roots of unity, the following theorem of Fischer guarantees that $k(G)$ is rational.

\begin{theorem}[Fischer {\cite[Theorem 6.1]{Sw2}}] \label{t1.5}
Let $G$ be a finite abelian group of exponent $e$, $k$ be a field containing $\zeta_e$, a primitive $e$-th root of unity.
For any finite-dimensional representation $G\to GL(V)$ over $k$,
the fixed field $k(V)^G$ is rational over $k$.
\end{theorem}

When $G$ is abelian and $k$ is any field (e.g. $k=\bm{Q}$), the rationality problem of $k(G)$ was investigated by Swan, Endo and Miyata, Voskresenskii, Lenstra, etc.. Swan's survey paper \cite{Sw2} gives an excellent account of  Noether's problem for abelian groups.

Now we turn to Noether's problem for non-abelian groups. We define the group $G_{m,n}$ first.

\begin{defn} \label{d1.2}
Let $m$ and $n$ be positive integers and assume that there is an integer $t$ such that $t\in (\bm{Z}/m\bm{Z})^\times$ is of order $n$. Define a group $G_{m,n}:=\langle\sigma,\tau:\sigma^m=\tau^n=1,\tau^{-1}\sigma\tau=\sigma^t\rangle$ $\simeq C_m \rtimes C_n$.

We remark that the integer $t$ always exists provided that $m$ is an odd prime power and $n \mid \phi(m)$. Also note that the group $G_{m,n}$ depends on the choice of $t$. However, for different choices of $t$, the rationality criterion we are concerned about (e.g. Theorem \ref{t4.7}) is not affected, which may be justified by applying Lemma \ref{l3.3}. Thus we will not emphasize the dependence of $G_{m,n}$ on the choice of $t$.
\end{defn}

We will study under what situation the fixed field $\bm{C}(G_{m,n})$ will be rational. The following theorem is quite useful, provided that the ring $\bm{Z}[\zeta_n]$ is a UFD (unique factorization domain).

\begin{theorem}[{\cite[Theorem 1.4]{Ka2}}] \label{t1.1}
Let $k$ be a field and $G$ be a finite group.
Assume that (i) $G$ contains an abelian normal subgroup $H$ so that $G/H$ is cyclic of order $n$,
(ii) $\bm{Z}[\zeta_n]$ is a UFD where $\zeta_n$ is a primitive $n$-th root of unity,
and (iii) $\zeta_e\in k$ where $e:=\fn{lcm}\{\fn{ord} (g):g\in G\}$ is the exponent of $G$.
If $G\to GL(V)$ is any finite-dimensional linear representation of $G$ over $k$,
then $k(V)^G$ is rational over $k$.
\end{theorem}

It is unknown for a long time whether $\bm{C}(G_{p,q})$ is rational or not if $p$ and $q$ are distinct prime numbers and $\bm{Z}[\zeta_q]$ is not a UFD. We note that the assumption that $\bm{Z}[\zeta_n]$ is a UFD in Theorem \ref{t1.1} imposes a severe restriction to the integer $n$, because of the theorem of Masley and Montgomery.

\begin{theorem}[Masley and Montgomery \cite{MM}] \label{t1.2}
$\bm{Z}[\zeta_n]$ is a unique factorization domain if and if $1\le n\le 22$,
or $n=24, 25, 26, 27, 28, 30, 32, 33, 34, 35, 36, 38, 40, 42, 44, 45, 48, 50, 54$, $60, 66, 70, 84, 90$.
\end{theorem}

A recent work of Chu and Huang \cite{CH} found a sufficient condition for the rationality of $\bm{C}(G_{m,q})$ where $q$ is a prime number.

\begin{theorem}[Chu and Huang {\cite[Main Theorem]{CH}}] \label{t1.4}
Let $m$ and $q$ be positive integers where $q$ is a prime number and assume that there is an integer $t$ such that $t\in(\bm{Z}/m\bm{Z})^\times$ is of order $q$. Define $m'=m/\gcd\{m,t-1\}$.
Assume that there exist integers $a_0,a_1,\ldots,a_{q-2},b$ such that $\gcd\{a_0,a_1,\ldots,a_{q-2},b\}=1$,
$bm'=a_0+a_1t+\cdots+a_{q-2}t^{q-2}$ and $N_{\bm{Q}(\zeta_q)/\bm{Q}}(\alpha)=m'$ where $\alpha:=a_0+a_1\zeta_q+\cdots+a_{q-2}\zeta_q^{q-2}$.
If $k$ is a field with $\zeta_m, \zeta_q \in k$, then $k(G_{m,q})$ is rational over $k$.

In particular, for distinct prime numbers $p$ and $q$ with $q \mid \phi(p)$, if there is an element $\alpha \in \bm{Z}[\zeta_q]$ such that $N_{\bm{Q}(\zeta_q)/\bm{Q}}(\alpha)=p$ and $k$ is a field with $\zeta_p, \zeta_q \in k$, then $k(G_{p,q})$ is rational over $k$.

\end{theorem}

As an application, Chu and Huang show that $\bm{C}(G_{p,q})$ is rational
when $(p,q)=(5801,29)$, $(6263,31)$, $(32783,37)$, $(101107,41)$;
for more examples, see Section 4 of \cite{CH}.
Note that $\bm{Z}[\zeta_q]$ is not a UFD when $q=29,31,37$ or $41$ by Theorem \ref{t1.2};
thus these examples escape the application of Theorem \ref{t1.1}.

We remark that the proof of Theorem \ref{t1.4} given in \cite{CH} is rather computational and lengthy. Moreover, the assumptions, e.g. the number $m'$, look abrupt at first sight.

The purpose of this paper is to provide a conceptual approach to the rationality of $k(G_{m,n})$ different from the computational verification in \cite{CH}. Moreover, a generalized form of Theorem \ref{t1.4} can be found in Theorem \ref{t4.6} and Theorem \ref{t4.7}. A clarification of the number $m'$ is given in Lemma \ref{l4.5}; we will show that the ``purpose" of the complicated assumptions in Theorem \ref{t1.4} is just to ensure that the ideal $\langle \zeta_q - t, m^{\prime} \rangle$ is a principal ideal of $\bm{Z}[\zeta_q]$ (see Step 3 in the proof of Theorem \ref{t4.6}). The examples constructed by computer computing in \cite[Section 4]{CH} turn out to be heralds of Theorem \ref{t4.8}, which asserts that, for any positive integer $n$, there are infinitely many prime numbers $p$ such that the fixed field $\bm{C}(G_{p,n})$ is rational over $\bm{C}$.

In this article we will study the rationality problem of $k(G_{m,n})$ where $n$ is any positive integer; we don't assume that $n$ is a prime number as in \cite{CH}. Here is a sample of our results.

\begin{theorem} \label{t1.6}
Let $m$ and $n$ be positive integers such that $m$ is an odd integer. Assume that (i) there is an integer $t$ satisfying that $t\in (\bm{Z}/m\bm{Z})^\times$ is of order $n$,
and (ii) for any $e\mid n$, the ideal $\langle \zeta_e-t,m\rangle$ in $\bm{Z}[\zeta_e]$ is a principal ideal.
If $k$ is a field with $\zeta_m, \zeta_n\in k$,
then $k(G_{m,n})$ is rational over $k$.
\end{theorem}

For other results, see Theorem \ref{t4.3} and Theorem \ref{t4.7}.

The main idea in the proofs of Theorem \ref{t1.6} and its variants is to apply the methods developed by Endo, Miyata, Lenstra etc. in solving Noether's problem for abelian groups \cite{EM1,EM2,Le}. These methods were reformulated by Colliot-Th\'el\`ene and Sansuc \cite[Section 1]{CTS} (see Section 2 for a brief summary). Armed with these tools, we will embark on the investigation of the rationality problem of $k(G_{m,n})$ in Section 4.

In Section 4, the reader will find that the rationality of $\bm{C}(G_{m,n})$ is reduced to the rationality of $\bm{C}(M)^\pi$
where $\pi \simeq C_n$ and $M$ is a $\pi$-lattice
(see Section 2 for the definition of a $\pi$-lattice and the multiplicative invariant field $\bm{C}(M)^\pi$).
Such a rationality problem was studied by Saltman \cite{Sa2}, Beneish and Ramsey \cite{BR}. One of the aims of Salman in \cite{Sa2} is to find a group $\pi$ and a $\pi$-lattice $M$ such that $\bm{C}(M)^\pi$ is not retract rational (and thus not stably rational); it is necessary that such a group $\pi$ is not cyclic by \cite{Ka3}. On the other hand, Beneish and Ramsey considered a cyclic group $\pi$ and proposed a notion, the Property $*$ for $\pi$ \cite[Definition 3.3]{BR}. Assuming the Property $*$, they were able to prove two significant results.

\begin{theorem}[Beneish and Ramsey {\cite[Theorem 3.12 and Theorem 3.13]{BR}}] \label{t1.7}
Assume that the Property $*$ for a cyclic group of order $n$ is valid.

(i) Let $\pi \simeq C_n$ and $M$ be any $\pi$-lattice. If $k$ is a field with $\zeta_n \in k$, then $k(M)^{\pi}$ is stably rational over $k$.

(ii) Let $G:=A \rtimes C_n$ where $A$ is a finite abelian group of exponent $e$. If $k$ is a field with $\zeta_e, \zeta_n \in k$, then $k(G)$ is stably rational over $k$.
\end{theorem}

In Theorem \ref{t5.2} we will prove that the Property $*$ for a cyclic group of order $n$ is equivalent to the assertion that $\bm{Z}[\zeta_n]$ is a UFD. As a result an alternative proof of Theorem \ref{t1.7} will be given in Section 5. In fact, similar results are valid, say, for the dihedral group, because such kind of theorems are consequences of Endo-Miyata's Theorem \cite[Theorem 3.3; EK, Theorem 1.4]{EM2}; see Theorem \ref{t5.3} and Lemma \ref{l5.4}.

Finally we remark that, besides the sufficient condition for the rationality of $\bm{C}(G_{m,n})$,
the retract rationality, one of the necessary conditions, is already known before.
For a field extension $L$ of $k$ (where $k$ is an infinite field),
the notion that $L$ is retract rational over $k$ is introduced by Saltman \cite{Sa1}.
It is known that ``rational" $\Rightarrow$ ``stably rational" $\Rightarrow$ ``retract rational" $\Rightarrow$ ``unirational".
The reader may consult Theorem 1.6, Theorem 1.7 and Theorem 1.8 of \cite{Ka3}
for a quick review of the retract rationality of $\bm{C}(A\rtimes C_n)$ where $A$ is a finite abelian group.

\begin{idef}{Standing notations.}
Throughout this article, we consider only finite groups. The following notations are adopted:

$C_n$: the cyclic group of order $n$, \par
$\pi$: a finite group, \par
$\bm{Z}\pi$: the integral group ring of $\pi$, \par
$\Phi_n(X)$: the $n$-th cyclotomic polynomial, \par
$\phi(n)$: the value of the Euler $\phi$-function at $n$, \par
$(\bm{Z}/m\bm{Z})^\times$: the group of units of the ring $\bm{Z}/m\bm{Z}$, \par
$\fn{ord}_p(n)$: the exponent of $p$ in $n$, i.e.\ if $\fn{ord}_p(n)=e$, then $p^e\mid n$ but $p^{e+1}\nmid n$, \par
$\zeta_n$: a primitive $n$-th root of unity.
\end{idef}

When we say that $\zeta_n\in k$ ($k$ is a field), it is understood that either $\fn{char}k=0$ or $\fn{char}k>0$ with $\fn{char}k\nmid n$.
Recall that a field extension $L/k$ is rational if $L$ is purely transcendental over $k$,
i.e.\ $L\simeq k(x_1,\ldots,x_n)$ over $k$ where $k(x_1,\ldots,x_n)$ is the rational function field of $n$ variables over $k$.
A field extension $L/k$ is stably rational if $L(y_1,\ldots,y_m)$ is rational over $k$
where $y_1,\ldots,y_m$ are some elements algebraically independent over $L$. If $\pi$ is a group and $\tau \in \pi$, then $\langle \tau \rangle$ denotes the subgroup generated by $\tau$; similarly, if $R$ is a commutative ring and $\alpha_1, \alpha_2, \cdots, \alpha_n \in R$, then $\langle \alpha_1, \alpha_2, \cdots, \alpha_n \rangle$ denotes the ideal generated by $\alpha_1, \alpha_2, \cdots, \alpha_n$.

Let $m$ and $n$ be positive integers.
For the sake of simplicity, we will simply say that $t\in (\bm{Z}/m\bm{Z})^\times$ is of order $n$,
when we mean that $t\in\bm{Z}$, $\gcd\{t,m\}=1$ and the subgroup $\langle \bar{t}\rangle \simeq C_n$
where $\bar{t}$ is the residue class of $t$ in $(\bm{Z}/m\bm{Z})^\times$.

Finally we remind the reader that the fixed field $k(G)$ is defined at the beginning of this section.

\bigskip
Acknowledgments. The proof of Theorem \ref{t4.8} was suggested by Prof. Ching-Li Chai (Univ. of Pennsylvania), whose help is highly appreciated.

\section{Preliminaries}

Let $\pi$ be a finite group.
We recall the definition of $\pi$-lattices.

\begin{defn} \label{d2.1}
Let $\pi$ be a finite group.
A finitely generated $\bm{Z}[\pi]$-module $M$ is called a $\pi$-lattice if $M$ is a free abelian group when it is regarded as an abelian group.

If $M$ is a $\pi$-lattice and $L$ is a field with $\pi$-action,
we will associate a rational function field over $L$ with $\pi$-action as follows.
Suppose that $M=\bigoplus_{1\le i\le m}\bm{Z}\cdot u_i$.
Define $L(M)=L(x_1,\ldots,x_m)$, a rational function field of $m$ variables over $L$.
For any $\sigma\in \pi$, if $\sigma\cdot u_i=\sum_{1\le j\le m} a_{ij}u_j$ in $M$ (where $a_{ij}\in \bm{Z}$),
we define $\sigma\cdot x_i=\prod_{1\le j\le m} x_j^{a_{ij}}$ in $L(M)$ and,
for any $\alpha\in L$, define $\sigma\cdot\alpha$ by the prescribed $\pi$-action on $L$.

Note that, if $\pi$ acts faithfully on $L$ and $K=L^\pi$ (i.e.\ $\pi\simeq \fn{Gal}(L/K)$),
then the fixed field $L(M)^\pi=\{f\in L(M):\sigma\cdot f=f$ for all $\sigma\in\pi\}$ is the function field of an algebraic torus defined over $K$,
split by $L$ and with character lattice $M$ (see \cite[Section 12; Sa2]{Vo,Sw2}).

On the other hand, if $\pi$ acts trivially on $L$ (i.e.\ $\sigma(\alpha)=\alpha$ for all $\sigma\in \pi$, for all $\alpha\in L$),
the action of $\pi$ on $L(M)$ is called a purely monomial action in some literature. When we write $k(M)^{\pi}$ without emphasizing the action of $\pi$ on $k$, it is understood that $\pi$ acts trivially on $k$, i.e. the situation of purely monomial actions.
\end{defn}

\begin{defn} \label{d2.2}
Let $\pi$ be a finite group and $M$ be a $\pi$-lattice.
$M$ is called a permutation lattice if $M$ has a $\bm{Z}$-basis permuted by $\pi$.
A $\pi$-lattice $M$ is called an invertible lattice if it is a direct summand of some permutation lattice.
A $\pi$-lattice $M$ is called a flabby lattice if $H^{-1}(\pi',M)=0$ for all subgroup $\pi'$ of $\pi$;
it is called a coflabby lattice if $H^1(\pi',M)=0$ for all subgroups $\pi'$ of $\pi$.
For the basic properties of $\pi$-lattices, see \cite{CTS,Sw2}.
\end{defn}

\begin{defn} \label{d2.3}
Let $\pi$ be a finite group.
Denote by $\c{L}_{\pi}$ (resp.\ $\c{F}_{\pi}$) the class of all the $\pi$-lattices (resp.\ all the flabby $\pi$-lattices).
We introduce a similarity relation on $\c{L}_{\pi}$ and $\c{F}_{\pi}$:
two lattices $M_1$ and $M_2$ are similar,
denoted by $M_1\sim M_2$, if $M_1\oplus Q_1\simeq M_2\oplus Q_2$ for some permutation $\pi$-lattices $Q_1$ and $Q_2$.
Let $\c{L}_{\pi}/{\sim}$ and $\c{F}_{\pi}/{\sim}$ be the sets of similarity classes of $\c{L}_{\pi}$ and $\c{F}_{\pi}$ respectively;
we define $F_{\pi}=\c{F}_{\pi}/{\sim}$.
For each $\pi$-lattice $M$,
denote by $[M]$ the similarity class containing $M$.

We define an addition on $\c{L}_{\pi}/{\sim}$ and $F_{\pi}$ as follows: $[M_1]+[M_2]:=[M_1\oplus M_2]$ for any $\pi$-lattices $M_1$ and $M_2$.
In this way, $\c{L}_{\pi}/{\sim}$ becomes an abelian monoid and $F_{\pi}$ is a submonoid of $\c{L}_{\pi}/{\sim}$.
Note that $[M]=0$ in $F_{\pi}$ if and only if $M$ is stably permutation,
i.e.\ $M\oplus Q$ is isomorphic to a permutation $\pi$-lattice where $Q$ is some permutation $\pi$-lattice.
See \cite{Sw2} for details.
\end{defn}

\begin{defn} \label{d2.4}
Let $\pi$ be a finite group, $M$ be a $\pi$-lattice.
The $M$ have a flabby resolution,
i.e.\ there is an exact sequence of $\pi$-lattices:
$0\to M\to Q\to E\to 0$ where $Q$ is a permutation lattice and $E$ is a flabby lattice \cite[Lemma 1.1; CTS; Sw2]{EM2}.

Although the above flabby resolution is not unique, the class $[E]\in F_{\pi}$ is uniquely determined by $M$.
Thus we define the flabby class of $M$, denoted as $[M]^{fl}$,
by $[M]^{fl}=[E]\in F_{\pi}$ (see \cite{Sw2}).
Sometimes we say that $[M]^{fl}$ is permutation or invertible if the class $[E]$ contains a permutation lattice or an invertible lattice.
\end{defn}

\begin{theorem} \label{t2.5}
Let $L/K$ be a finite Galois extension with $\pi=\fn{Gal}(L/K)$,
and let $M$ be a $\pi$-lattice.
The group $\pi$ acts on the field $L(M)$ as in Definition \ref{d2.1}.
\begin{enumerate}
\item[$(1)$]
{\rm (\cite[Theorem 1.6; Vo; Le, Theorem 1.7; CTS]{EM1})}
The fixed field $L(M)^\pi$ is stably rational over $K$ if and only if $[M]^{fl}=0$ in $F_{\pi}$.

\item[$(2)$]
{\rm (\cite[Theorem 3.14]{Sa1})}
Assume that $K$ is an infinite field.
Then the fixed field $L(M)^\pi$ is retract rational over $K$ if and only if $[M]^{fl}$ is invertible.
\end{enumerate}
\end{theorem}

Finally we recall a variant of the No-Name Lemma.

\begin{theorem}[{\cite[Theorem 1]{HK}}] \label{t2.6}
Let $L$ be a field and $G$ be a finite group acting on $L(x_1,\ldots,x_m)$,
the rational function field of $m$ variables over $L$.
Assume that
\begin{enumerate}
\item[(i)]
for any $\sigma\in G$, $\sigma(L)\subset L$,
\item[(ii)]
the restriction of the action of $G$ to $L$ is faithful, and
\item[(iii)]
for any $\sigma\in G$,
\[
\begin{pmatrix} \sigma(x_1) \\ \vdots \\ \sigma(x_m) \end{pmatrix}
=A(\sigma)\begin{pmatrix} x_1 \\ \vdots \\ x_m \end{pmatrix}+B(\sigma)
\]
where $A(\sigma)\in GL_m(L)$ and $B(\sigma)$ is an $m\times 1$ matrix over $L$.
\end{enumerate}
Then $L(x_1,\ldots,x_m)=L(z_1,\ldots,z_m)$ for some elements $z_1,\ldots,z_m \in L(x_1,\ldots,x_m)$
such that $\sigma(z_i)=z_i$ for all $\sigma\in G$, for all $1\le i\le m$.
Consequently, $L(x_1,\ldots,x_m)^G=L^G(z_1,\ldots,z_m)$.
\end{theorem}

\section{Some ideals of $\bm{Z}[\zeta_n]$}

Let $\bm{Z}[\zeta_n]$ be the ring of integers of the cyclotomic field $\bm{Q}(\zeta_n)$.

\begin{lemma} \label{l3.1}
{\rm (1) (\cite[Lemma 1; Le, Lemma 3.8]{Ar})}
Let $p\ge 3$ be a prime number.
For simplicity, we write $\fn{ord}(m)$ for $\fn{ord}_p(m)$.
Suppose that $t\in (\bm{Z}/p\bm{Z})^\times$ is of order $n$, then
\begin{quote}
$\fn{ord}(\Phi_n(t))=\fn{ord}(t^n-1)\ge 1$; \par
$\fn{ord}(\Phi_{p^d n} (t))=1$ if $d\ge 1$; \par
$\fn{ord}(\Phi_e(t))=0$ if $e\in\bm{N}$ and $e$ is not of the form $p^d n$ (where $d\ge 0$); \par
$\fn{ord}(t^l-1)=0$ if $l\in\bm{N}$ and $n\nmid l$; \par
$\fn{ord}(t^l-1)=\fn{ord}(t^n-1)+\fn{ord}(l)$ if $l\in\bm{N}$ and $n\mid l$.
\end{quote}

{\rm (2) (\cite[page 15]{EM1})}
Let $p\ge 3$ and $\fn{ord}(m)$ be the same as in $(1)$.

Let $d\ge 1$, $n\mid \phi(p^d)$ and $n=p^{d_0}n_0$ where $0\le d_0\le d-1$ and $p\nmid n_0$.
Let $t\in (\bm{Z}/p^d\bm{Z})^\times$ be of order $n$.
Then
\begin{quote}
$\fn{ord}(\Phi_{n_0}(t))\ge d$ if $d_0=0$, i.e.\ $p\nmid n$; \par
$\fn{ord}(\Phi_{n_0}(t))=d-d_0$ if $d_0\ge 1$; \par
$\fn{ord}(\Phi_{p^{d'} n_0}(t))=1$ if $d_0\ge 1$ and $1\le d'\le d_0$, \par
$\fn{ord}(\Phi_{p^{d'} n'}(t))=0$ if $0\le d'\le d_0$, $n'\mid n_0$ and $n'<n_0$.
\end{quote}

{\rm (3) (\cite[Lemma 1]{Ar})}
$\fn{ord}_2(\Phi_{2^d}(t))=1$ if $d\ge 2$ and $t$ is an odd integer.
\end{lemma}

\begin{lemma} \label{l3.2}
Let $m_1$, $m_2$ and $n$ be positive integer such that $\gcd\{m_1,m_2\}=1$.
If $J$ is an ideal of $\bm{Z}[\zeta_n]$,
then $\langle J,m_1m_2\rangle=\langle J,m_1\rangle\cdot \langle J,m_2\rangle=\langle J,m_1\rangle \cap \langle J,m_2\rangle$.
\end{lemma}

\begin{proof}
Since $m_1$ and $m_2$ are relatively prime, the ideal $\langle J,m_1 \rangle$ and $\langle J,m_2\rangle$ are comaximal. Thus $\langle J,m_1\rangle\cdot \langle J,m_2\rangle = \langle J,m_1\rangle \cap \langle J,m_2\rangle$

It is clear that $\langle J, m_1m_2\rangle \subset \langle J,m_1\rangle \cap \langle J, m_2\rangle
=\langle J,m_1\rangle\cdot \langle J,m_2\rangle \subset \langle J,m_1m_2\rangle$.
Done.
\end{proof}

Recall that, if $J$ is an ideal of $\bm{Z}[\zeta_n]$,
then the (absolute) norm of $J$, denoted by $N_{\bm{Q}(\zeta_n)/\bm{Q}}(J)$, is the index of $J$ in $\bm{Z}[\zeta_n]$;
in other words, $N_{\bm{Q}(\zeta_n)/\bm{Q}}(J)=|\bm{Z}[\zeta_n]/J|$
(see, for examples, \cite[pages 203--204]{IR}).
Consequently, if $a\in \bm{Z}[\zeta_n]\backslash \{0\}$,
then $N_{\bm{Q}(\zeta_n)/\bm{Q}}$ $(\langle a\rangle)=|\fn{Norm}_{\bm{Q}(\zeta_n)/\bm{Q}} (a)|$.

\begin{lemma} \label{l3.3}
Let $m=p^d$ where $p\ge 3$ is a prime number and $d\ge 1$.
Let $n$ be a positive integer with $n\mid \phi(m)$.
Write $n=p^{d_0}n_0$ where $d_0\ge 0$ and $p\nmid n_0$.
Suppose that $t\in (\bm{Z}/m\bm{Z})^\times$ is of order $n$.

$(1)$ If $e\mid n$ and $e=p^{d'}n'$ with $0\le d'\le d_0$ and $p\nmid n'$, then
\[
\langle \zeta_e-t,p\rangle =\begin{cases}
\text{a proper ideal of $\bm{Z}[\zeta_e]$, if }n'=n_0; \\ \bm{Z}[\zeta_e], \text{ otherwise}.
\end{cases}
\]

Moreover, for $1\le d'' \le d$, $\langle \zeta_e-t,p^{d''}\rangle=\langle \zeta_e-t,p\rangle ^{d''}$.

$(2)$ Every prime ideal of $\bm{Z}[\zeta_n]$ lying over $p$ is of the form $\langle \zeta_n-t',p\rangle$
where $t'$ is an integer and $t' \in (\bm{Z}/m\bm{Z})^\times$ is of order $n$. All of these prime ideals are conjugate in $\bm{Z}[\zeta_n]$. In fact, if $e\mid n$ and $1\le d'' \le d$, the ideals $\langle \zeta_e-t,p^{d''}\rangle$ and $\langle \zeta_e-t',p^{d''}\rangle$ are conjugate in $\bm{Z}[\zeta_e]$.

$(3)$ The following assertions are equivalent:
The ideal $\langle \zeta_n-t,p\rangle$ is a principal ideal
$\Leftrightarrow$ There is an element $\alpha\in \bm{Z}[\zeta_n]$ such that $N_{\bm{Q}(\zeta_n)/\bm{Q}}(\alpha)=\pm p$
$\Leftrightarrow$ For any $e\mid n$, the ideal $\langle \zeta_e-t,p\rangle $ is a principal ideal.
\end{lemma}

\begin{proof}
(1) If $e\mid n$ with $e=p^{d'}n'$ where $n'=n_0$, we will show that $\langle \zeta_e-t,p\rangle \subsetneq \bm{Z}[\zeta_e]$.
Otherwise, there are $x,y\in\bm{Z}[\zeta_e]$ such that $x(\zeta_e-t)+py=1$.
For any $g\in\fn{Gal}(\bm{Q}(\zeta_e)/\bm{Q})$, we have $g(x)\cdot (g(\zeta_e)-t)+pg(y)=1$.
Hence $1=\prod_g [g(x)(g(\zeta_e)-t)+pg(y)]=\alpha\cdot \Phi_e(t)+p\cdot \beta$ for some $\alpha,\beta\in \bm{Z}[\zeta_e]$.
By Lemma \ref{l3.1} (2), we have $p\mid \Phi_e(t)$.
Thus $p\mid \alpha \Phi_e(t)+p\beta$ in $\bm{Z}[\zeta_e]$.
A contradiction.

Now suppose that $e\mid n$ with $n'<n_0$.
By Lemma \ref{l3.1} (2) again, we find that $p\nmid \Phi_e(t)$.
Thus we may find integers $a$ and $b$ with $a\Phi_e(t)+pb=1$.
Since $\Phi_e(t)=\prod_i (t-\zeta_e^i)$ where $i$ runs over integers in $(\bm{Z}/e\bm{Z})^\times$,
it follows that $\langle \zeta_e-t,p\rangle=\bm{Z}[\zeta_e]$.

We will prove that $\langle \zeta_e-t,p^{d''}\rangle=\langle \zeta_e-t,p\rangle ^{d''}$. If $\langle \zeta_e-t,p\rangle= \bm{Z}[\zeta_e]$, then $\alpha (\zeta_e-t)+p\beta=1$ for some $\alpha,\beta \in \bm{Z}[\zeta_e]$. Hence $1=(\alpha (\zeta_e-t)+p\beta)^{d''}=\gamma (\zeta_e-t)+p^{d''}\beta^{d''} \in \langle \zeta_e-t,p ^{d''}\rangle$ for some $\gamma \in \bm{Z}[\zeta_e]$. It remains to consider the case when $\langle \zeta_e-t,p\rangle$ is a proper ideal of $\bm{Z}[\zeta_e]$. The ideals $\langle \zeta_e-t,p\rangle ^{d''}$ and $\langle \zeta_e-t,p^{d''}\rangle$ are of the same norm $p^{d''}$. Since $\langle \zeta_e-t,p\rangle ^{d''} \subset \langle \zeta_e-t,p^{d''}\rangle$, thus they are equal.

\medskip
(2) By (1), the ideal $\langle \zeta_n-t,p\rangle$ is of norm $p$;
thus it is a prime ideal lying over $p$. By \cite[page 182, Proposition 12.3.3]{IR} all the prime ideals of $\bm{Z}[\zeta_n]$ lying over $p$ are conjugate to each other. Thus they are of the form $\langle \zeta_n^u-t,p\rangle$ where $u$ is an integer, $1 \le u \le n-1$ and gcd$\{n, u \}=1$. For each $u$, choose an integer $v$ with $1 \le v \le n-1$ and $uv \equiv 1 (\bmod \, n)$. We will show that $\langle \zeta_n^u-t,p\rangle = \langle \zeta_n-t^v,p\rangle$. Since $\zeta_n-t^v=\zeta_n^{uv}-t^v$ is divisible by $\zeta_n^u -t$, we find that $\langle \zeta_n^u-t,p\rangle \supset \langle \zeta_n-t^v,p\rangle$. These two ideals are of the same norm $p$ by (1) (note that $\bm{Z}[\zeta_n]=\bm{Z}[\zeta_n^u]$). Thus they are equal. We conclude that all the prime ideals over $p$ are conjugate and are of the required form.

We will show that $\langle \zeta_e-t,p^{d''}\rangle$ and $\langle \zeta_e-t',p^{d''}\rangle$ are conjugate. If $\langle \zeta_e-t,p\rangle= \bm{Z}[\zeta_e]$, then $\langle \zeta_e-t',p\rangle= \bm{Z}[\zeta_e]$ because it depends only on the factorizations of $n$ and $e$ by (1). Using (1) again, we find that $\langle \zeta_e-t,p^{d''}\rangle=\bm{Z}[\zeta_e]=\langle \zeta_e-t',p^{d''}\rangle$. Now assume that both $\langle \zeta_e-t,p\rangle$ and $\langle \zeta_e-t',p\rangle$ are proper ideal of $\bm{Z}[\zeta_e]$. Then they are prime ideals lying over $p$. Hence they are conjugate in $\bm{Z}[\zeta_e]$. By (1) we find that $\langle \zeta_e-t,p^{d''}\rangle$ and $\langle \zeta_e-t',p^{d''}\rangle$ are also conjugate.

\medskip
(3) If $\alpha\in \bm{Z}[\zeta_n]$ is of norm $\pm p$, then $|\bm{Z}[\zeta_n]/\langle \alpha \rangle |=p$. Thus $\langle \alpha \rangle$ is a prime ideal and $p \in \langle \alpha \rangle$. Hence $\langle \alpha \rangle$ is a prime ideal over $p$. It follows that $\langle \alpha\rangle=\langle \zeta_n-t',p\rangle$ for some $t'\in (\bm{Z}/m\bm{Z})^\times$ of order $n$. Since $\langle \zeta_n-t,p\rangle$ and $\langle \zeta_n-t',p\rangle$ are conjugate by (2), we find that $\langle \zeta_n-t,p\rangle$ is a principal ideal.

Now assume that $\langle \zeta_n-t,p\rangle$ is a principal ideal.
For any $e\mid n$, consider $\langle \zeta_e-t,p\rangle$.
If $\langle \zeta_e-t,p\rangle=\bm{Z}[\zeta_e]$, there is nothing to prove.
Now assume that $\langle \zeta_e-t,p\rangle \subsetneq \bm{Z}[\zeta_e]$.
Then $\langle \zeta_e-t,p\rangle$ is a prime ideal lying over $p$.
Since there is an element $\alpha\in \bm{Z}[\zeta_n]$ with $N_{\bm{Q}(\zeta_n)/\bm{Q}} (\alpha)=\pm p$,
define $\beta=N_{\bm{Q}(\zeta_n)/\bm{Q}(\zeta_e)} (\alpha)$.
Then $N_{\bm{Q}(\zeta_e)/\bm{Q}}(\beta)=\pm p$.
It follows that $\langle \beta\rangle$ is a prime ideal of $\bm{Z}[\zeta_e]$ over $p$.
But $\langle \zeta_e-t,p\rangle$ is also a prime ideal over $p$.
Thus $\langle \zeta_e-t,p\rangle$ is conjugate to $\langle \beta \rangle$.
Hence $\langle \zeta_e-t,p\rangle$ is a principal ideal of $\bm{Z}[\zeta_e]$.
\end{proof}

\section{The main results}

Recall that, for positive integers $m$, $n$ and the integer $t\in (\bm{Z}/m\bm{Z})^\times$ of order $n$, the group $G_{m,n}$ is defined in Definition \ref{d1.2}. In this section, We will find sufficient conditions to guarantee that $\bm{C}(G_{m,n})$ is rational over $\bm{C}$ under various situations for $m$ and $n$.

\begin{theorem}[{\cite[Theorem 1.11; Vo; Le, Theorem 2.6]{EM1}}] \label{t4.1}
Let $K/k$ be a finite Galois extension with $\pi=\fn{Gal}(K/k)$.
Let $M$ be a $\pi$-lattice.

Assume furthermore that $\pi=\langle \tau\rangle \simeq C_n$ and $M$ is a projective $\bm{Z}\pi$-module.
Then the following three statements are equivalent:
\begin{enumerate}
\item[(i)]
the field $K(M)^\pi$ is stably rational over $k$;
\item[(ii)]
the field $K(M)^\pi$ is rational over $k$;
\item[(iii)]
for any $e\mid n$, the module $M/\Phi_e(\tau)M$ is a free $\bm{Z}[\zeta_e]$-module.
(Note that $M/\Phi_e(\tau)M$ is a module over $\bm{Z}\pi/\Phi_e(\tau)\simeq \bm{Z}[\zeta_e]$.)
\end{enumerate}
\end{theorem}

Before stating Theorem \ref{t4.2}, we explain some terminology.
We denote by $C(\bm{Z}[\zeta_e])$ the ideal class group of $\bm{Z}[\zeta_e]$,
i.e.\ the quotient group of the group of non-zero fractional ideals of $\bm{Z}[\zeta_e]$ by the subgroup of principal ideals. Thus, if $M$ is a finitely generated torsion-free $\bm{Z}[\zeta_e]$-module and
$M\simeq I_1\oplus I_2\oplus \cdots \oplus I_l$ where each $I_j$ is a non-zero ideal of $\bm{Z}[\zeta_e]$,
then the class of $M$ in $C(\bm{Z}[\zeta_e])$,
denoted by $[M]$, is defined as $[M]=[I_1\cdot I_2\cdots I_l]\in C(\bm{Z}[\zeta_e])$.
If $M$ is a finitely generated $\bm{Z}[\zeta_e]$-module,
$(M)_0$ denotes $M/t(M)$ where $t(M)$ is the torsion submodule of $M$.

\begin{theorem}[{\cite[page 86; Sw3, Theorem 2.10; EK, Theorem 1.4]{EM2}}] \label{t4.2}
Let $\pi=\langle\tau\rangle \simeq C_n$, $F_\pi$ be the flabby class monoid of Definition \ref{d2.3}.
Then $F_\pi$ is a finite group and the group homomorphism $c:F_\pi \to \bigoplus_{e\mid n} C(\bm{Z}[\zeta_e])$
defined by $c([M])=(\ldots,[(M/\Phi_e(\tau)M)_0],\ldots)$ is an isomorphism where $M$ is a flabby $\pi$-lattice.
\end{theorem}

\begin{theorem} \label{t4.3}
Let $m$ and $n$ be positive integers.
Assume that (i) $\gcd\{m,n\}=1$, there is an integer $t$ such that $t\in (\bm{Z}/m\bm{Z})^\times$ is of order $n$, and (ii) for any $e\mid n$, the ideal $\langle \zeta_e-t,m\rangle$ in $\bm{Z}[\zeta_e]$ is a principal ideal. If $k$ is a field with $\zeta_m,\zeta_n \in k$, then $k(G_{m,n})$ is rational over $k$.
\end{theorem}

\begin{proof}
The situation $m=1$ or $n=1$ is trivial.
Thus we may assume that $m,n\ge 2$. For simplicity, write $G=G_{m,n}$ and $G=\langle\sigma,\tau:\sigma^m=\tau^n=1,\tau^{-1}\sigma\tau=\sigma^t\rangle$.

\medskip
Step 1.
Let $V:=V_{reg}=\bigoplus_{g\in G} k\cdot u_g$ be the regular representation space of $G$ such that $h\cdot u_g=u_{hg}$ for any $g,h\in G$.
Let $\{x(g):g\in G\}$ be the dual basis of $\{u_g:g\in G\}$.
Then the induced action of $G$ acts on the dual space $V^*=\bigoplus_{g\in G} k\cdot x(g)$ by $h\cdot x(g)=x(hg)$ for any $g,h\in G$. Since $k[V]$ is the symmetric algebra of $V^*$, we find that $k(G)=k(V)^G=k(x(g):g\in G)^G$.

\medskip
Define $X=\sum_{0\le i\le m-1} \zeta_m^{-i} x(\sigma^i)\in V^*$. Then $\sigma\cdot X=\zeta_m X$.

Define $y_j=\tau^j\cdot X\in V^*$ for $0\le j\le n-1$.

It is easy to verify that
\begin{align*}
\sigma &: y_j \mapsto \zeta_m^{t^j} y_j \text{ for }0\le j\le n-1, \\
\tau &: y_0\mapsto y_1\mapsto \cdots\mapsto y_{n-1}\mapsto y_0.
\end{align*}

By Theorem \ref{t2.6}, $k(G)=k(y_j:0\le j\le n-1)^G (z_1,\ldots,z_{mn-n})$ where $g\cdot z_i=z_i$ for all $1\le i\le mn-n$ and for all $g\in G$.

\medskip
Now $k(y_j:0\le j\le n-1)^{\langle \sigma\rangle}=k(y_0^m,y_j/y_{j-1}^t: 1\le j\le n-1)$.

Let $\pi=\langle \tau\rangle$ and $\langle y_j:0\le j\le n-1\rangle$ be the multiplicative subgroup of $k(y_j:0\le j\le n-1)\backslash \{0\}$.
As a $\pi$-lattice, $\langle y_j:0\le j\le n-1\rangle \simeq \bm{Z}\pi$.
Define a $\pi$-sublattice $M$ of $\bm{Z}\pi$ by
\begin{equation}
M=\langle \tau-t,m \rangle \subset \bm{Z}\pi  \label{eq4.1}
\end{equation}
Under the isomorphism $\langle y_j:0\le j\le n-1\rangle \simeq \bm{Z}\pi$, $M$ corresponds to $\langle y_0^m, y_j/y_{j-1}^t: 1\le j\le n-1\rangle$.
In other words, $k(y_0^m,y_j/y_{j-1}^t:1\le j\le n-1)\simeq k(M)$ and
$k(G)=k(y_j: 0\le j\le n-1)^G(z_1,\ldots,z_{mn-n})\simeq k(M)^{\langle\tau\rangle}(z_1,\ldots,z_{mn-n})$.

The $\pi$-lattice $M$ is called the Masuda's ideal in \cite[page 14]{EM1}. We will show that it is a projective ideal of $\bm{Z}\pi$ in Step 4, i.e. a left ideal of of $\bm{Z}\pi$ which is also a $\bm{Z}\pi$-projective module.

\medskip
Step 2.
Let $M'=\bigoplus_{0\le j\le n-1} \bm{Z}\cdot v_j$ be the $\pi$-lattice defined by
$\tau: v_0\mapsto v_1 \mapsto \cdots \mapsto v_{n-1}\mapsto v_0$,
i.e.\ $M'\simeq \bm{Z}\pi$.
Consider $k(M\oplus M')^{\langle \tau\rangle}$.

Since $\pi=\langle\tau\rangle$ acts faithfully on $k(M)$ and $k(M\oplus M')=k(M)(v_j:0\le j \le n-1)$,
it follows that $k(M\oplus M')^{\langle \tau\rangle}=k(M)^{\langle \tau\rangle}(w_1,\ldots,w_n)$
where $\tau\cdot w_j=w_j$ for all $1\le j\le n$ by Theorem \ref{t2.6}.
It follows that $k(G)\simeq k(M\oplus M')^{\langle\tau\rangle}(z'_1,z'_2,\ldots,z'_{mn-2n})$.

We will show that $k(M\oplus M')^{\langle\tau\rangle}$ is rational over $k$.
Once this is finished, we find that $k(G)$ is rational over $k$.

\medskip
Step 3.
Regard $k(M\oplus M')^{\langle\tau\rangle}=k(M')(M)^{\langle\tau\rangle}$.
Write $K:=k(M')$, $k_0=k(M')^{\langle \tau\rangle}$.
Then $K/k_0$ is a finite Galois extension with $\fn{Gal}(K/k_0)=\pi$.
Moreover, $k_0=k(M')^{\langle\tau\rangle}$ is rational over $k$ by Theorem \ref{t1.5}
since $\zeta_n\in k$ and $\pi=\langle\tau\rangle$ is abelian.
If we show that $K(M)^{\langle\tau\rangle}=k(M')(M)^{\langle\tau\rangle}$ is rational over $k_0$,
then $k(M\oplus M')^{\langle\tau\rangle}$ is rational over $k$.

\medskip
Step 4.
From the definition of $M$, i.e.\ Formula \eqref{eq4.1}, it is clear that $[\bm{Z}\pi:M]=m$.
Since $\gcd\{m,n\}=1$, it follows that $M$ is a projective $\bm{Z}\pi$-module by \cite[Proposition 7.1; Ka2, Theorem 3.9]{Sw1}. Hence we may apply Theorem \ref{t4.1} to $M$.

We remark that this is the only situation in which the assumption $\gcd\{m,n\}=1$ is used.

\medskip
Note that, since $M$ is $\bm{Z}\pi$-projective, it follows that $M/\Phi_e (\tau)M=\bm{Z}\pi/\Phi_e(\tau)\otimes M$ is torsion-free,
i.e.\ $M/\Phi_e(\tau)M=(M/\Phi_e(\tau)M)_0$ in the notation of Theorem \ref{t4.2}.

Since $M$ is a sublattice of $\bm{Z}\pi$ with $\bm{Z}\pi/M$ torsion,
we may use \cite[Proposition 2.2]{Le} to evaluate $M/\Phi_e(\tau)M$.
We find that, $M/\Phi_e(\tau)M$ is the image of $M=\langle \tau-t,m\rangle$ in $\bm{Z}\pi/\Phi_e(\tau)\cdot \bm{Z}\pi\simeq \bm{Z}[\zeta_e]$.
We find that $M/\Phi_e(\tau)M\simeq \langle \zeta_e-t,m\rangle$.

By assumption, for all $e\mid n$, $\langle\zeta_e-t,m\rangle$ is a principal ideal,
i.e.\ the class of $\langle \zeta_e-t,m\rangle$ in $C(\bm{Z}[\zeta_e])$ is the zero class.
By Theorem \ref{t4.2},
we obtain that $c([M])=0$ and thus $[M]=0$ in $F_\pi$.

In conclusion, $M$ is a projective ideal of $\bm{Z}\pi$ and $[M]=0$ in $F_{\pi}$
(equivalently, there is permutation $\pi$-lattices $Q_1$ and $Q_2$ such that $M\oplus Q_1\simeq Q_2$).

\medskip
Step 5.
Since $M$ is $\bm{Z}\pi$-projective,
there is a projective $\bm{Z}\pi$-module $P$ such that $M\oplus P\simeq \bm{Z}\pi^{(l)}$ for some integer $l$.
It follows that $0\to M\to \bm{Z}\pi^{(l)}\to P\to 0$ is a flabby resolution.
Thus $[M]^{fl}=[P]$.

From $M\oplus P\simeq \bm{Z}\pi^{(l)}$ and $M\oplus Q_1\simeq Q_2$,
we get $Q_2\oplus P\simeq Q_1\oplus M\oplus P\simeq Q_1\oplus \bm{Z}\pi^{(l)}$.
Hence $[M]^{fl}=[P]=0$ in $F_\pi$.

By Theorem \ref{t2.5}, $K(M)^{\langle\tau\rangle}$ is stably rational over $k_0$ (remember that the definitions of $K$ and $k_0$ in Step 3).
Apply Theorem \ref{t4.1}.
We obtain that $K(M)^{\langle\tau\rangle}$ is rational over $k_0$ because $M$ is $\bm{Z}\pi$-projective.
\end{proof}

\begin{remark}
In the above theorem, the assumption $\zeta_m, \zeta_n \in k$ may be replaced by $\zeta_m \in k$ and $k(C_n)$ is rational over $k$, because we may apply Theorem \ref{t1.5}.
\end{remark}

\begin{lemma} \label{l4.12}
Let $\pi$ be a cyclic group of order $n$ and $M$ be a $\pi$-lattice satisfying that $[M]^{fl}=0$. If $k$ is a field with $\zeta_n\in k$, then $k(M)^{\pi}$ is stably rational over $k$.
\end{lemma}

\begin{proof}
Step 1. Define $\pi^{\prime}=\{\lambda \in \pi: \lambda$ acts trivially on $M \}$. Define $\pi^{\prime \prime}=\pi/\pi^{\prime}$. Then $M$ is a faithful $\pi^{\prime \prime}$-lattice. As a $\pi$-lattice, $[M]^{fl}=0$. It follows that, as a $\pi^{\prime \prime}$-lattice, we also have $[M]^{fl}=0$ by \cite[page 180, Lemma 2]{CTS}.

\medskip
Step 2. We use the ideas of Step 2 and Step 3 in the proof of Theorem \ref{t4.3}.

Define $M^{\prime}=\bm{Z}\pi^{\prime \prime}$. Consider the fixed field $k(M \oplus M^{\prime})^{\pi} (=k(M \oplus M)^{\pi^{\prime \prime}})$. Since $M$ is a faithful $\pi^{\prime \prime}$-lattice, we may apply Theorem \ref{t2.6}. We get $k(M \oplus M^{\prime})^{\pi}=k(M)^{\pi}(w_1, \ldots, w_e)$ where $e= |\pi^{\prime \prime}|$.

On the other hand, since $[M]^{fl}=0$ as a $\pi^{\prime \prime}$-lattice, we may apply Theorem \ref{t2.5}. It follows that $k(M \oplus M^{\prime})^{\pi}$ is stably rational over $k(M^{\prime})^{\pi^{\prime \prime}}$. Since $\zeta_n \in k$, clearly $\zeta_e \in k$ (remember that $e= |\pi^{\prime \prime}|$). Hence $k(M^{\prime})^{\pi^{\prime \prime}}$ is rational over $k$ by Theorem \ref{t1.5}. In conclusion, $k(M \oplus M^{\prime})^{\pi}$ is stably rational over $k$. Thus we find that $k(M)^{\pi}$ is also stably rational over $k$.
\end{proof}

\begin{remark}
In the above lemma, the condition $[M]^{fl}=0$ is a mild restriction. In fact, it may happen that $M$ may be a projective $\bm{Z}\pi$-module which is not stably free and is not a permutation lattice, while $[M]=0$ (equivalently, $[M]^{fl}=0$). Here is such an example taken from \cite[page 18, line -17]{EM1}.

Let $\pi=\langle \tau\rangle \simeq C_{12}$,
$t$ be an integer such that $t\in (\bm{Z}/13\bm{Z})^\times$ is of order 12, i.e. a primitive root of $(\bm{Z}/13\bm{Z})^\times$.
Then $M=\langle \tau-t,13\rangle \subset \bm{Z}\pi$ is a projective $\bm{Z}\pi$-module by \cite[Proposition 2.6]{EM1}. $M$ is a projective ideal of $\bm{Z}\pi$.

Clearly $M$ is not a free module, i.e. $M$ is not a principal ideal of $\bm{Z}\pi$. Obviously it is not a permutation $\pi$-lattice. We claim that $M$ is not a stably free $\bm{Z}\pi$-module. Otherwise, we have $M \oplus \bm{Z}\pi^{(s)} \simeq \bm{Z}\pi^{(s+1)}$ for some positive integer $s$. Taking the the determinant of both sides, we find that $M$ is a free module, which is a contradiction.

However, the class $[M]\in F_\pi$ is the zero class because $c([M])=0$ by Theorem \ref{t4.2}
(by Theorem \ref{t1.2} $\bm{Z}[\zeta_e]$ is a UFD for all divisors $e$ of $12$).
It follows that $M$ is stably permutation.
\end{remark}

\begin{lemma} \label{l4.4}
Let $m$, $n$ be positive integers.
Assume that (i) $m=p^d$ ($p\ge 3$ is a prime number, $d\ge 1$),
$t$ is an integer such that $t\in (\bm{Z}/m\bm{Z})^\times$ is of order $n$,
and (ii) for any $e\mid n$,
the ideal $\langle \zeta_e-t,m\rangle$ in $\bm{Z}[\zeta_e]$ is a principal ideal.
If $k$ is a field with $\zeta_m,\zeta_n\in k$, then $k(G_{m,n})$ is rational over $k$.
\end{lemma}

\begin{proof}
Write $G_{m,n}=\langle\sigma,\tau:\sigma^m=\tau^n=1,\tau^{-1}\sigma\tau=\sigma^t\rangle$.

All the assumptions and the conclusion in this lemma are the same as those in Theorem \ref{t4.3},
except that we don't assume that $\gcd\{m,n\}=1$ and replace it by $m=p^d$.

The proof of Theorem \ref{t4.3} remains valid till Step 3.
We will also show that the ideal $M=\langle \tau-t,m\rangle\subset \bm{Z}\pi$ is a projective $\bm{Z}\pi$-module.

The fact that $M$ is $\bm{Z}\pi$-projective follows from \cite[Proposition 2.6]{EM1}. In fact, it is the Masuda's ideal mentioned in Step 1 of the proof of Theorem \ref{t4.3}. The ideal $M$ is denoted by $M_V$ (and also by $I_k(p^{l_j})$) in \cite[page 14--15]{EM1}.

For the convenience of the reader, we explain briefly the main idea of the proof \cite[Proposition 2.6]{EM1}. Using our notation, consider the module $\bm{Z}\pi/M$. By a direct verification, we show that $\bm{Z}\pi/M$ is cohomologically trivial \cite[Theorem 4.12]{Ri} (the verification is straightforward). Hence $M$ is also cohomologically trivial. Then we may apply \cite[Theorem 4.11]{Ri} to conclude that $M$ is projective.

Once we know that $M$ is $\bm{Z}\pi$-projective, the remaining proof of Theorem \ref{t4.3} works as before.
\end{proof}

\begin{proof}[Proof of Theorem \ref{t1.6}] --------------------~\par

Let $G:=G_{m,n}=\langle\sigma,\tau:\sigma^m=\tau^n=1,\tau^{-1}\sigma\tau=\sigma^t\rangle$.

Write $m=\prod_{1\le i\le r} p_i^{d_i}$ where $p_1,\ldots,p_r$ are distinct prime number and $d_i\ge 1$.

Define $m_i:=m/p_i^{d_i}$ for $1\le i\le r$;
and define $\sigma_i=\sigma^{m_i}$, $H_i=\langle \sigma_j:j\ne i\rangle$.
Then $\langle \sigma\rangle=\langle \sigma_i\rangle \times H_i$, and $\langle\sigma\rangle=\langle\sigma_1,\ldots,\sigma_r\rangle$.

\medskip
Let $V^*=\bigoplus_{g\in G} k\cdot x(g)$ be the same as in Step 1 of the proof of Theorem \ref{t4.3}.
For $1\le j\le r$, define
\[
X_j=\sum_{g\in H_j}\sum_{0\le l\le p_j^{d_j}-1} \zeta_{p_j^{d_i}}^{-l} x(g\sigma_j^l)\in V^*,
~ y_l^{(j)}=\tau^l X_j \quad\text{for }0\le l\le n-1.
\]

It is routine to check that
\begin{align*}
\sigma_i:{}& X_i \mapsto \zeta_{p_i^{d_i}} X_i,~X_j\mapsto X_j \quad\text{if } j\ne i, \\
& y_l^{(i)} \mapsto \zeta_{p_i^{d_i}}^{t^l} \cdot y_l^{(i)},~y_l^{(j)}\mapsto y_l^{(j)} \quad\text{if } j\ne i \\
\tau:{} & y_0^{(i)}\mapsto y_1^{(i)}\mapsto \cdots\mapsto y_{n-1}^{(i)} \mapsto y_0^{(i)} \quad\text{for }1\le i\le r.
\end{align*}

By Theorem \ref{t2.6}, $k(G)$ is rational over $k(y_l^{(i)}:1\le i\le r, 0\le l\le n-1)^G$.
Moreover, $k(y_l^{(i)}:1\le i\le r,0\le l\le n-1)^{\langle\sigma\rangle}=
k\big((y_0^{(i)})^{p_i^{d_i}}, y_l^{(i)}/(y_{l-1}^{(i)})^t:1\le i\le r,1\le l\le n-1 \big)$.

\medskip
Write $\pi=\langle\tau\rangle$.
Define a $\pi$-lattice $M_i$ of $\bm{Z}\pi$ by
\begin{equation}
M_i=\langle \tau-t,p_i^{d_i}\rangle \quad\text{for } 1\le i\le r. \label{eq4.2}
\end{equation}

Since each $p_i$ is odd, it follows that $M_i$ is $\bm{Z}\pi$-projective by \cite[Proposition 2.6]{EM1} (see the proof of Lemma \ref{l4.4}).
Note that $k(y_l^{(i)}:1\le i\le r,0\le l\le n-1)^G=k(M_1\oplus\cdots\oplus M_r)^{\langle\tau\rangle}$.

For any $e\mid n$, $\langle\zeta_e-t,m\rangle=\prod_{1\le i\le r} \langle\zeta_e-t,p_i^{d_i}\rangle=\prod_{1\le i\le r} [M_i/\Phi_e(\tau)M_i]$
is the $e$-th component of $c(M)$ where $M=\bigoplus_{1\le i\le r}M_i$ and $c:F_\pi \to \bigoplus_{e\mid n} C(\bm{Z}[\zeta_e])$
is the isomorphism defined in Theorem \ref{t4.2}.

By assumption $\langle \zeta_e-t,m\rangle$ is a principal ideal.
Hence $c(M)=0$ and thus $[M]=0$ in $F_\pi$ by Theorem \ref{t4.2}.
The remaining part of the proof is the same as in Theorem \ref{t4.3}.
Done.
\end{proof}

In Theorem \ref{t4.6} we will give a generalization of Theorem \ref{t1.4} using the method of Theorem \ref{t4.3}.

Recall the assumptions in Theorem \ref{t1.4}: $q$ is a prime number, $t\in (\bm{Z}/m\bm{Z})^\times$ is of order $q$,
$G_{m,n}=\langle\sigma,\tau:\sigma^m=\tau^q=1,\tau^{-1}\sigma\tau=\sigma^t\rangle$,
and $m'=m/\gcd\{m,t-1\}$.

Write $m=q^{d_0}\prod_{1\le i\le s} p_i^{d_i}\cdot \prod_{1\le j\le s'} q_j^{e_j}$ where $d_0=\fn{ord}_q(m)\ge 0$, $d_i,e_j\ge 1$,
$q$, $p_i$, $q_j$ are distinct prime numbers such that $p_i\nmid t-1$ for $1\le i\le s$,
and $q_j\mid t-1$ for $1\le j\le s'$.

Define $m''=\prod_{1\le i\le s}p_i^{d_i}$. Define $m_1$ and $m_2$ by the formula $m=m_1m_2$ and $m_2$ is defined by
\begin{equation}
m_2=\begin{cases}
q^{d_0}m'' & \text{if }q\mid m', \\
m'' & \text{if } q\nmid m'.
\end{cases} \label{eq4.3}
\end{equation}

\begin{lemma} \label{l4.5}
Let $m'$, $m''$, $m_1$ and $m_2$ be defined as above.
Then
\[
m'=\begin{cases}
qm'' & \text{if } 1\le \fn{ord}_q (t-1)<\fn{ord}_q(m), \\ m'' & \text{otherwise}.
\end{cases}
\]

In fact, $q\mid m'$ if and only if $\fn{ord}_q(t-1) =d_0-1\ge 1$ with $\fn{ord}_q(m_2)=d_0$.
\end{lemma}

\begin{proof}
If $q\nmid m$, then $q\nmid m'$.
Thus we will consider the situation $q\mid m$ in the sequel.
We will show that, if $1\le \fn{ord}_q (t-1)<\fn{ord}_q(m)$, then $\fn{ord}_q(t-1)=\fn{ord}_q(m)-1$.

Since $q\mid m$, it follows that $t^q\equiv 1 \pmod{q^{d_0}}$ where $d_0=\exp_q(m)$. From $t\equiv t^q {\pmod{q}}$, we get $q\mid t-1$.

If $\fn{ord}_q (t-1)\ge \fn{ord}_q(m)$, then $q\nmid m'$ and $q\nmid m_2$.
It remains to consider the situation $1\le \fn{ord}_q (t-1)<\fn{ord}_q(m)$.

Write $t=1+ag^{d'}$ and $m=bq^{d_0}$ where $1\le d' <d_0$ and $q\nmid ab$.
From $t^q\equiv 1 \pmod{q^{d_0}}$,
we find that $t^q-1=(1+aq^{d'})^q-1=aq^{d'+1}+cq^{d'+2}$ for some integer $c$.
Thus $d'+1=d_0$ as we expected.
Clearly $\fn{ord}_q(m')=1$ and $\fn{ord}_q(m_2)=d_0$.

\medskip
Suppose that $q_j\ne q$ is a prime divisor of $m$ with $q_j\mid t-1$.
Write $t=1+aq_j^{d'}$, $m=bq_j^{e_j}$ where $d',e_j\ge 1$ and $q_j\nmid ab$.
We claim that $d'\ge e_j$.
Otherwise, $d'\le e_j-1$.
Then $q_j^{e_j}$ will not divide $t^q-1=(1+aq_j^{d'})^q-1=aqq_j^{d'}+cq_j^{2d'}$ where $c$ is some integer.
We conclude that $d'\ge e_j$ and $q_j\nmid m'$.
\end{proof}

The following theorem is slightly different from Theorem \ref{t1.4}. In fact, we don't require that $\gcd\{a_0,a_1,\ldots,a_{q-2},b\}=1$ and the positivity condition of the norm in Theorem \ref{t1.4} is waived.

\begin{theorem} \label{t4.6}
Let $m$ and $q$ be positive integers where $q$ is a prime number and assume that there is an integer $t$ such that $t\in (\bm{Z}/m\bm{Z})^\times$ is of order $q$.
Define $m'=m/\gcd\{m,t-1\}$.
Assume that there exist integers $a_0,a_1,\ldots,a_{q-2},b$ such that $bm'=a_0+a_1t+\cdots+a_{q-2}t^{q-2}$ and
$N_{\bm{Q}(\zeta_q)/\bm{Q}} (\alpha)=\pm m'$ where $\alpha:=a_0+a_1\zeta_q+\cdots+a_{q-2}\zeta_q^{q-2}$.
If $k$ is a field with $\zeta_m, \zeta_q\in k$,
then $k(G_{m,n})$ is rational over $k$.
\end{theorem}

\begin{proof}

Write $G:=G_{m,n}=\langle\sigma,\tau:\sigma^m=\tau^q=1,\tau^{-1}\sigma\tau=\sigma^t\rangle$.

\medskip
Step 1.
If $q=2$, we may apply Theorem \ref{t1.1}.
Thus we assume that $q\ge 3$ from now on.
The notations $m'$, $m''$, $m_1$, $m_2$ and $p_i$ ($1\le i\le s$), $q_j$ ($1\le j\le s'$) in Lemma \ref{l4.5} remain in force throughout the proof.

Note that $m''$ and $m_2$ are odd integers.
Otherwise, $p_i=2$ for some $1\le i\le s$.
By definition, $t\in (\bm{Z}/p_i^{d_i}\bm{Z})^\times$ is of order $q>1$.
If $p_i=2$, then $(\bm{Z}/p_i^{d_i}\bm{Z})^\times$ is a group of order $2^{d_i-1}$.
Since $q$ is odd, we get a contradiction.

\medskip
Step 2.
By Lemma \ref{l4.5} and the definitions of $m_1$ and $m_2$ in Equation (4.3), it is routine to verify that $m_1 \mid t-1$ no matter whether $q \mid m^{\prime}$ or not.

Define $\lambda=\sigma^{m_2}$, $\rho=\sigma^{m_1}$. Then $\tau^{-1}\lambda \tau=\lambda$ because $m_1 \mid t-1$.

Define $H:=\langle \lambda \rangle \simeq C_{m_1}$ and $G_0:= \langle \rho, \tau \rangle \simeq G_{m_2,q}$.

Thus $G=H\times G_0$. If both $k(H)$ and $k(G_0)$ are rational, then $k(G)$ is also rational by \cite[Theorem 1.3]{KP}. Note that $k(H)$ is rational by Theorem \ref{t1.5}. It remains to show that $k(G_0)$ is rational over $k$.

To show that $k(G_0)$ is rational over $k$, we will apply Theorem \ref{t1.6} (remember that both $m_2$ and $q$ are odd by Step 1). Thus the goal is to show that $\langle \zeta_q-t,m_2\rangle$ is a principal ideal of $\bm{Z}[\zeta_q]$.

\medskip
Step 3.
From the assumption of Theorem \ref{t1.4} there exist integers $a_0,a_1,\ldots,a_{q-2},b$ such that $bm'=a_0+a_1t+\cdots+a_{q-2}t^{q-2}$
and $N_{\bm{Q}(\zeta_q)/\bm{Q}} (\alpha)=\pm m'$ where $\alpha:=a_0+a_1\zeta_q+\cdots+a_{q-2}\zeta_q^{q-2}$.
It follows that $\alpha-bm'=(\zeta_q-t)\cdot \beta$ for some $\beta \in \bm{Z}[\zeta_q]$,
i.e.\ $\alpha\in \langle \zeta_q-t,m'\rangle$.

We will show that $N_{\bm{Q}(\zeta_q)/\bm{Q}} (\langle \zeta_q-t,m'\rangle)=m'$.

If $q\nmid m'$, then $m'=m''=\prod_i p^d$ and $\langle \zeta_q-t,m'\rangle=\prod_i \langle\zeta_q-t,p_i^{d_i}\rangle$ by Lemma \ref{l3.2}.
Since $\langle \zeta_q-t,p_i^{d_i}\rangle \subset \langle\zeta_q-t,p_i\rangle$ and $t\in (\bm{Z}/p_i^{d_i} \bm{Z})^\times$ is of order $q$,
it follows that $\langle\zeta_q-t,p_i\rangle \subsetneq \bm{Z}[\zeta_q]$ is of index $p_i$ by Lemma \ref{l3.3}.
Hence $N_{\bm{Q}(\zeta_q)/\bm{Q}} (\langle\zeta_q-t,p_i\rangle)=p_i$ and $N_{\bm{Q}(\zeta_q)/\bm{Q}}(\langle\zeta_q-t,p_i^{d_i}\rangle)=p_i^{d_i}$.
Thus $N_{\bm{Q}(\zeta_q)/\bm{Q}}(\langle\zeta_q-t,m'\rangle)=m'$.

For the other situation, if $q\mid m'$, then $m'=qm''$.
We have $\langle\zeta_q-t,m'\rangle=\langle\zeta_q-t,q\rangle\cdot \langle \zeta_q-t,m''\rangle$.
Note that $t-1=aq^{d_0-1}$ where $d_0-1\ge 1$ and $q\nmid a$ by Lemma \ref{l4.5}.
It follows that $\zeta_q-t=\zeta_q-1-(t-1)=(\zeta_q-1)\cdot (1+(\zeta_q-1)\cdot \beta)$ for some $\beta\in\bm{Z}[\zeta_q]$.
Thus $\langle \zeta_q-t,q\rangle=\langle \zeta_q-1\rangle$ is of norm $q$.
We can show that $N_{\bm{Q}(\zeta_q)/\bm{Q}}(\langle\zeta_q-t,m''\rangle)=m''$ as the previous situation.
Hence we find $N_{\bm{Q}(\zeta_q)/\bm{Q}}(\langle \zeta_q-t,m'\rangle)=m'$.

Since $\langle\alpha\rangle\subset \langle\zeta_q-t,m'\rangle$ and both ideals have the same norm $m'$,
it follows that $\langle\alpha\rangle=\langle\zeta_q-t,m'\rangle$,
i.e.\ the ideal $\langle\zeta_q-t,m'\rangle$ is a principal ideal.

\medskip
Step 4.
We will show that the ideals $\langle\zeta_q-t,m''\rangle$ and $\langle\zeta_q-t,m_2\rangle$ are principal ideals.
This will finish the proof by Step 2.

If $q\nmid m'$, then $m'=m''=m_2$. Thus $\langle\zeta_q-t,m_2\rangle= \langle\zeta_q-t,m''\rangle =\langle\zeta_q-t,m'\rangle$ is a principal ideal by Step 3.

It remains to consider the case $q\mid m'$, i.e. $m'=qm''$ and $m_2=q^{d_0}m''$.

From $\langle\alpha\rangle=\langle\zeta_q-t,m'\rangle =\langle\zeta_q-t,q\rangle\langle\zeta_q-t,m''\rangle
=\langle\zeta_q-1\rangle\cdot\langle\zeta_q-t,m''\rangle$, we get $\alpha=(\zeta_q-1)\cdot \beta$ for some $\beta\in \langle\zeta_q-t,m''\rangle$.
Hence $\langle\zeta_q-t,m''\rangle=\langle\beta\rangle$ is a principal ideal.

Now $\langle\zeta_q-t,m_2\rangle=\langle \zeta_q-t,q^{d_0}\rangle\cdot \langle\zeta_q-t,m''\rangle
=\langle\zeta_q-1\rangle \cdot \langle\zeta_q-t,m''\rangle$ is a principal ideal because so is the ideal $\langle\zeta_q-t,m''\rangle$.
\end{proof}

\begin{theorem} \label{t4.7}
Let $p$ and $n$ be positive integers where $p$ is a prime number and assume that $t$ is an integer such that $t\in (\bm{Z}/p\bm{Z})^\times$ is of order $n$.
Assume that there is some element $\alpha\in \bm{Z}[\zeta_n]$ satisfying that $N_{\bm{Q}(\zeta_n)/\bm{Q}}(\alpha)=\pm p$.
If $k$ is a field with $\zeta_p,\zeta_n\in k$, then $k(G_{p,n})$ is rational over $k$.
\end{theorem}

\begin{proof}
The ideals $\langle\alpha\rangle$ and $\langle\zeta_n-t,p\rangle$ are of norm $p$.
Both of them are prime ideals of $\bm{Z}[\zeta_n]$ lying over $p$.
Thus they are conjugate to each other.
In particular, the ideal $\langle\zeta_n-t,p\rangle$ is a principal ideal.
For any $e\mid n$, if $e <n$, then $\langle\zeta_e-t,p\rangle=\bm{Z}[\zeta_e]$ by Lemma \ref{l3.3}.
Now apply Theorem \ref{t4.3}.
\end{proof}

\begin{theorem} \label{t4.8}
Let $n$ be a positive integer.
Define $S:=\{p\in \bm{N}:p$ is a prime number and $p$ splits completely into the product of principal prime ideals of $\bm{Z}[\zeta_n]\}$, and define $S_0=\{p\in \bm{N}:p$ is a prime number such that $\bm{C}(G_{p,n})$ is rational over $\bm{C} \}$. Then the Dirichlet densities of $S$ and $S_0$ are positive; in particular, $S_0$ is an infinite set. Consequently, there are infinitely many prime numbers $p$ satisfying that $\bm{C}(G_{p,n})$ is rational over $\bm{C}$.
\end{theorem}

\begin{proof}
Step 1.
If $p \in S$, then $p$ splits completely in $\bm{Z}[\zeta_n]$. Thus $n \mid \phi(p)$ by considering the factorization of $\Phi_n(X)$$\pmod{p}$; alternatively, apply \cite[page 103]{Ne}.
It follows that there is an integer $t$ such that $t\in (\bm{Z}/p\bm{Z})^\times$ is of order $n$. Hence we may define the group $G_{p,n}$ as in Definition \ref{d1.2}.

Since  $p$ splits completely into the product of principal prime ideals of $\bm{Z}[\zeta_n]$, there is some element  $\alpha\in\bm{Z}[\zeta_n]$ such that $\langle\alpha\rangle$ is a prime ideal lying over $p$ with  $N_{\bm{Q}(\zeta_n)/\bm{Q}}(\alpha)=\pm p$ by Lemma \ref{l3.3}. It follows that $\bm{C}(G_{p,n})$ is rational by Theorem \ref{t4.7}. In summary, if $p$ is a prime number and $p \in S$, then $\bm{C}(G_{p,n})$ is rational, i.e. $S \subset S_0$.

It remains to show that the Dirichlet density of $S$ is positive. We denote by $d(S)$ the Dirichlet density of $S$; the reader is referred to \cite[page 130]{Ne} for its definition.

\medskip
Step 2.
Suppose that $p \in S$ and $P$ is a non-zero principal prime ideal of $\bm{Z}[\zeta_n]$ lying over $p$,
then $P$ is of degree one, i.e.\ natural map $\bm{Z}/p\bm{Z} \to \bm{Z}[\zeta_n]/P$ is an isomorphism.
Define $T_1=\{P:P$ is a non-zero principal prime ideal of degree one in $\bm{Z}[\zeta_n]\}$, and $T=\{P:P$ is a non-zero principal prime ideal of $\bm{Z}[\zeta_n]\}$. We will show that $d(T_1)=d(T)=h_n$ where $h_n$ is the class number of $\bm{Q}(\zeta_n)$.

\medskip
Step 3.
Let $L$ be the Hilbert class field of $\bm{Q}(\zeta_n)$; note that $\fn{Gal}(L/\bm{Q}(\zeta_n))\simeq C(\bm{Z}[\zeta_n])$ the ideal class group of $\bm{Z}[\zeta_n]$. If $P$ is a non-zero prime ideal of $\bm{Z}[\zeta_n]$, denote by $(P,L/\bm{Q}(\zeta_n))$ the Artin symbol of $P$, if $P$ is unramified in $L$ (see \cite[page 105]{Ne}). If $P$ is a non-zero prime ideals of $\bm{Z}[\zeta_n]$, then $P\in T$ if and only if $P$ splits completely in $L$ by \cite[page 107, Corollary 8.5]{Ne} (alternatively, this fact is one of the conditions in the definition of the Hilbert class field). On the other hand, $P$ splits completely in $L$ is equivalent to $(P,L/\bm{Q}(\zeta_n))=1$. It follows from the Tchebotarev density theorem that $d(T)=1/h_n$ \cite[page 132, Theorem 6.4]{Ne}.

In summary, we have
\begin{equation}
\frac{1}{h_n}=\lim_{s\to 1^+} \frac{\sum_{P\in T}\frac{1}{N(P)^s}}{\sum_P\frac{1}{N(P)^s}} \label{eq4.4}
\end{equation}
where $N(P)$ is the abbreviation of $N_{\bm{Q}(\zeta_n)/\bm{Q}}(P)$.

\medskip
Step 4.
Note that $\sum_P \frac{1}{N(P)^s} \sim \log \frac{1}{s-1}$ (see \cite[page 130]{Ne}).

Also note that $\sum_{P\in T}\frac{1}{N(P)^s} \sim \sum_{P\in T \atop \fn{deg} P=1} \frac{1}{N(P)^s}$
because $\sum_{\fn{deg}P\ge 2} \frac{1}{N(P)^s}$ is an analytic function at $s=1$ (see \cite[page 130]{Ne}).
We find that
\begin{equation}
\frac{1}{h_n}=\lim_{s\to 1^+}\frac{\sum_{P\in T_1}\frac{1}{N(P)^s}}{\log \frac{1}{s-1}}. \label{eq4.5}
\end{equation}

\medskip
Step 5.
Define a function $\Psi:T_1\to S$ by $\Psi(P)=P\cap \bm{Z}$ (here we identify a prime number $p$ with the prime ideal $p\bm{Z}$ in $\bm{Z}$). Note that $\Psi$ is well-defined. Also note that $\Psi^{-1}(p)$ is a set of $\phi(n)$ elements for any $p\in S$.
Thus
\[
\sum_{P\in T_1}\frac{1}{N(P)^s} = \phi(n)\cdot \sum_{p\in S} \frac{1}{p^s}.
\]

We have also $\sum_p \frac{1}{p^s}\sim \log \frac{1}{s-1}$ as in Step 4.
From Formula \eqref{eq4.5}, we obtain
\[
d(S)=\lim_{s\to 1^+} \frac{\sum_{p\in S}\frac{1}{p^s}}{\sum_p \frac{1}{p^s}}
=\lim_{s\to 1^+} \frac{1}{\phi(n)} \frac{\sum_{P\in T_1}\frac{1}{N(P)^s}}{\log\frac{1}{s-1}}=\frac{1}{\phi(n)\cdot h_n}. \qedhere
\]
\end{proof}

\begin{example} \label{ex4.9}
We will supply examples of groups $G\simeq A\rtimes C_n$ (where $A$ is an abelian group)
such that $\bm{C}(G)$ is rational while we cannot apply Theorem \ref{t1.1} to assert the rationality (because $\bm{Z}[\zeta_n]$ is not a UFD).

Let $d_1,\ldots,d_s$, $d$ be integers with $1\le d\le \min\{d_1-1,d_2-1,\ldots,d_s-1\}$.
Define $n=3^d$, $m_i=3^{d_i}$ for $1\le i\le s$.
For each $1\le i\le s$, find an integer $t_i$ such that $t_i\in (\bm{Z}/m_i\bm{Z})^\times$ is of order $n$.
Define $G:=\langle\sigma_1,\ldots,\sigma_s,\tau:\sigma_i^{m_i}=\tau^n=1,\sigma_i\sigma_j=\sigma_j\sigma_i$,
$\tau^{-1}\sigma_i\tau=\sigma_i^{t_i}$ for $1\le i\le s\rangle \simeq (C_{m_1}\times C_{m_2}\times\cdots\times C_{m_s})\rtimes C_n$.
Note that $G$ is a 3-group.
If $\zeta_{3^e}\in k$ (where $e\ge d_i$ for all $1\le i\le s$),
we claim that $k(G)$ is rational over $k$.

Note that, if $d\ge 4$, then $\bm{Z}[\zeta_n]$ is not a UFD by Theorem \ref{t1.2}.
Thus we cannot apply Theorem \ref{t1.1} to show that $k(G)$ is rational if $d\ge 4$.

On the other hand, the method in Lemma \ref{l4.4} and in the proof of Theorem \ref{t1.6} may be adapted to show that $k(G)$ is rational.
Write $\pi=\langle\tau\rangle$ and define $\pi$-lattices $M_i$ as in Formula \eqref{eq4.2} by
\[
M_i=\langle \tau-t_i,m_i\rangle \subset \bm{Z}\pi.
\]

As before, it is easy to see that each $M_i$ is a $\bm{Z}\pi$-projective module and $k(G)$ is rational over $k(M_1\oplus\cdots\oplus M_s)^{\langle\tau\rangle}$.
By the same proof as \cite[Corollary 3.3]{EM1},
$[M_i]^{fl}=0$ for all $i$ (rigorously speaking, $[M_i]^{fl}=0$ because of \cite[Proposition 3.2]{EM1} and its proof).
Thus $[M_1\oplus M_2\oplus \cdots\oplus M_s]^{fl}=0$ and the same arguments in the proof of Theorem \ref{t1.6} work as well.
Done.

The same method can be applied to 2-groups.
Let $d_1,\ldots,d_s$, $d$ be integers with $1\le d\le \min\{d_1-2,d_2-2,\ldots,d_s-2\}$.
Define $n=2^d$, $m_i=2^{d_i}$.
Find an integer $t_i$ such that $t_i\in (\bm{Z}/m_i\bm{Z})^\times$ is of order $n$
(this is possible because $d\le d_i-2$).
Define $G:=\langle\sigma_1,\ldots,\sigma_s,\tau:\sigma_i^{m_i}=\tau^n=1,\sigma_i\sigma_j=\sigma_j\sigma_i,
\tau^{-1}\sigma_i\tau=\sigma_i^{t_i}$ for $1\le i\le s\rangle$.
If $\zeta_{2^e}\in k$ (where $e\ge d_i$ for all $1\le i\le s$), we claim that $k(G)$ is rational over $k$.

Note that, if $d\ge 6$, then $\bm{Z}[\zeta_n]$ is not a UFD by Theorem \ref{t1.2}.
The proof of the rationality is similar to the above situation of 3-groups, but some modification is necessary.

If $d=1$, we may apply Theorem \ref{t1.1}.
Thus we assume that $d\ge 2$.
Since $\langle t_i\rangle\simeq C_n$ is a cyclic group of order $\ge 4$,
it follows that $-1\notin \langle t_i\rangle \subset (\bm{Z}/m_i\bm{Z})^\times$.
By \cite[Proposition 2.6]{EM1}, the $\pi$-lattice $\langle \tau-t_i,m_i\rangle \subset \bm{Z}\pi$ is projective where $\pi=\langle\tau\rangle$.
Note that $t-1$ is an even integer.
If $e\mid n$ and $e\ge 2$, then $\langle\zeta_e-t_i,m_i\rangle=\langle \zeta_e-1\rangle$.
Thus $[M_1\oplus M_2\oplus\cdots\oplus M_s]^{fl}=0$.
Done.

We remark that, if $G$ is a metacyclic $p$-group (a $p$-group which is an extension of a cyclic group by another cyclic group), then $\bm{C}(G)$ is always rational by \cite{Ka1}.
\end{example}

\begin{example} \label{ex4.10}
Let $p\ge 3$ be a prime number, $a_1,a_2,\ldots,a_s$ be positive integers with $p\nmid a_1a_2\cdots a_s$.
Define $m_i=a_ip^{d_i}$, $t_i=1+a_ip^{d_i-1}$ where $d_i\ge 2$ for $1\le i\le s$.
Define $G=\langle \sigma_1,\ldots,\sigma_s,\tau:\sigma_i^{m_i}=\tau^p=1,\sigma_i\sigma_j=\sigma_j\sigma_i,\tau^{-1}\sigma_i\tau=\sigma_i^{t_i}$
for $1\le i\le s\rangle$.
If $\zeta_e\in k$ (where $e=\fn{lcm}\{m_i:1\le i\le s\}$), we claim that $k(G)$ is rational.

Note that the case when $s=1$ is proved in \cite[Corollary 18]{CH}.

\medskip
Define $\lambda_i=\sigma_i^{p^{d_i}}$, $\rho_i=\sigma_i^{a_i}$,
$H_1=\langle\lambda_i:1\le i\le s\rangle$, $H_2=\langle\rho_i,\tau:1\le i\le s\rangle$.
Then $G=H_1\times H_2$.
Note that $k(H_1)$ is rational by Theorem \ref{t1.5}.
We will prove that $k(H_2)$ is rational.
Then we may apply \cite[Theorem 1.3]{KP} to conclude that $k(G)$ is rational.

Now we consider $k(H_2)$.
Define $\pi=\langle\tau\rangle$ and $\pi$-lattices $M_i$ defined by $M_i=\langle \tau-t_i$, $p^{d_i}\rangle$ as before.
Note that $\zeta_p-t_i=\zeta_p-1-(t_i-1)=(\zeta_p-1)(1+\alpha(\zeta_p-1)^{p-2})$ for some $\alpha\in \bm{Z}[\zeta_p]$.
Hence $\langle \zeta_p-t_i,p^{d_i}\rangle=\langle \zeta_p-1\rangle$ is a principal ideal.
Thus $[M_1\oplus\cdots\oplus M_s]^{fl}=0$.
Done.
\end{example}

\begin{example} \label{ex4.12}
We will give an application of Theorem \ref{t4.8}.

Let $k$ be an algebraic number field. Define $P_0= \{p: p$ is a prime numer, $p \le 43 \} \cup \{61, 67, 71 \}$ and $P_k=\{p: p$ ramifies in $k \}$. It is known that, if $p$ is a prime number, then ($1$) $\bm{Q}(C_p)$ is rational over $\bm{Q}$ if and only if $p \in P_0$ \cite{Pl}, and ($2$) if $p \notin (P_0 \cup P_k)$, then $k(C_p)$ is not stably rational over $k$ \cite{Ka4}. If $p \in P_0$, then $k(C_p)$ is rational over $k$ because $\bm{Q}(C_p)$ is rational over $\bm{Q}$. However, it is not clear whether $k(C_p)$ is rational over $k$ if $p \in P_k$.

We will construct infinitely many pairs $(p,k)$ where $p$ is a prime number, $k$ is an algebraic number field such that $p$ ramifies in $k$ and $k(C_p)$ is rational over $k$.

Given a positive integer $n$, let $S$ be the set defined in Theorem \ref{t4.8}. For each $p \in S$, $n \mid \phi(p)$ (see Step 1 in the proof of Theorem \ref{t4.8}). Choose $k$ to be the subfield of $\bm{Q}(\zeta_p)$ such that $[\bm{Q}(\zeta_p): k]=n$. Note that the field $k$ depends on the choice of $p \in S$. Also note that $p$ ramifies in $\bm{Q}(\zeta_p)$ (and also in $k$). Thus $\bm{Q}(\zeta_p)$ and $\bm{Q}(\zeta_n)$ are linearly disjoint over $\bm{Q}$, because $p$ splits completely in $\bm{Q}(\zeta_n)$.

Since $k(\zeta_p)=\bm{Q}(\zeta_p)$ is of degree $n$ over $k$ and there is some element  $\alpha\in\bm{Z}[\zeta_n]$ with  $N_{\bm{Q}(\zeta_n)/\bm{Q}}(\alpha)=\pm p$ (see Step 1 in the proof of Theorem \ref{t4.8}), we find that $k(C_p)$ is rational over $k$ by \cite[page 321, Corollary 7.1]{Le}.
\end{example}

\section{The multiplicative invariant fields}

Let $\pi$ be the cyclic group of order $n$.
Beneish and Ramsey \cite{BR} introduced a notion the Property $*$ for $\pi$ in \cite[Definition 3.3]{BR} and proved Theorem \ref{t1.7}.

Here is our interpretation of the Property $*$ for $\pi$.

\begin{theorem} \label{t5.2}
Let $\pi$ be a cyclic group of order $n$.
Then the following are equivalent:

(i) the Property $*$ for $\pi$;

(ii) $F_\pi=0$;

(iii) $\bm{Z}[\zeta_n]$ is a UFD.
\end{theorem}

\begin{proof}
(iii) $\Rightarrow$ (i) by \cite[Corollary 3.9]{BR}.

(i) $\Rightarrow$ (ii). By \cite[Lemma 2.2]{EK} we have an isomorphism $T(\pi)\to F_\pi$ where $T(\pi)$ is defined in \cite[page 86; EK, Definition 1.3]{EM2}.
In order to show that $F_\pi=0$,
it is necessary and sufficient to show that $T(\pi)=0$.
By \cite[Theorem 1.4]{EK}, we have an isomorphism $C(\bm{Z}\pi)/C^q(\bm{Z}\pi)\to T(\pi)$ where $C(\bm{Z}\pi)$ is the subgroup of $K_0(\bm{Z}\pi)$,
the Grothendireck group of the category of finitely generated projective $\bm{Z}\pi$-modules,
defined by $C(\bm{Z}\pi):=\{[\c{A}]-[\bm{Z}\pi]\in K_0(\bm{Z}\pi):\c{A}$ is a projective ideal over $\bm{Z}\pi\}$ (see \cite[Definition 2.11]{EK}).
The map $\varphi$ is induced by the map $\psi:C(\bm{Z}\pi)\to T(\pi)$ defined by sending $[\c{A}]-[\bm{Z}\pi]\in C(\bm{Z}\pi)$ to $[\c{A}]\in T(\pi)$;
note that $C^q(\bm{Z}\pi):=\{[\c{A}]-[\bm{Z}\pi]\in C(\bm{Z}\pi): [\c{A}]^{fl}=0$ in $F_\pi\}$
$=\{[\c{A}]-[\bm{Z}\pi]\in C(\bm{Z}\pi):[\c{A}]=0$ in $T(\pi)\}$ is the kernel of $\psi$ (see \cite[Definition 2.12]{EK}).
By \cite[Theorem 3.11]{BR} every projective ideal $\c{A}$ is stably permutation;
thus $C(\bm{Z}\pi)/C^q(\bm{Z}\pi)=0$ and so $T(\pi)=0$.

(ii) $\Rightarrow$ (iii) by Theorem \ref{t4.2}. Done.

We indicate a direct proof of (iii) $\Rightarrow$ (ii). Suppose that $\bm{Z}[\zeta_n]$ is a UFD.
For any $e\mid n$.
$\bm{Z}[\zeta_e]$ is also a UFD by \cite[page 39, Proposition 4.11]{Wa}.
Hence $\bigoplus_{e\mid n} C(\bm{Z}[\zeta])=0$.
By Theorem \ref{t4.2} again we find $F_\pi=0$.
\end{proof}

\begin{proof}[Proof of Theorem \ref{t1.7}] --------------------~\par
(i) By Theorem \ref{t5.2}, $F_\pi=0$. Thus $[M]^{fl}=0$ for any $\pi$-lattice $M$. Apply Lemma \ref{l4.12}.

(ii) By Theorem \ref{t5.2}, $\bm{Z}[\zeta_n]$ is a UFD. Apply Theorem \ref{t1.1}. Note that we prove that $k(G)$ is not only stably rational, but also rational.
\end{proof}

\begin{theorem} \label{t5.3}
Let $\pi \simeq D_n$ the dihedral group of order $2n$ where $n$ is an odd integer. Then $F_\pi=0$ if and only if $\bm{Z}[\zeta_n + \zeta_n^{-1}]$ is a UFD.
\end{theorem}

\begin{proof}
As in the proof of Theorem \ref{t5.2}, let $T(\pi)$ be defined in \cite[page 86; EK, Definition 1.3]{EM2}. Then we have $F_{\pi} \simeq T(\pi) \simeq C(\Omega_{\bm{Z}\pi})$ where $\Omega_{\bm{Z}\pi}$ is a maximal $\bm{Z}$-order in $\bm{Q}\pi$ containing $\bm{Z}\pi$ and $C(\Omega_{\bm{Z}\pi})$ is the class group of $\Omega_{\bm{Z}\pi}$ \cite[Theorem 3.3; EK, Theorem 1.4]{EM2}. It can be shown that $C(\Omega_{\bm{Z}\pi}) \simeq \oplus_{d \mid n} C(\bm{Z}[\zeta_d + \zeta_d^{-1}])$ (see, for examples, \cite[Theorem 6.3]{EK}).

If $F_{\pi}=0$, then $C(\bm{Z}[\zeta_n + \zeta_n^{-1}])=0$.

On the other hand, if $C(\bm{Z}[\zeta_n + \zeta_n^{-1}])=0$, then $C(\bm{Z}[\zeta_d + \zeta_d^{-1}])=0$ for all $d \mid n$ by \cite[page 39, Proposition 4.11]{Wa}. Thus $F_{\pi}=0$.
\end{proof}

\begin{lemma} \label{l5.4}
Let $\pi \simeq D_n$ the dihedral group of order $2n$ where $n$ is an odd integer. Assume that $\bm{Z}[\zeta_n + \zeta_n^{-1}]$ is a UFD.

(i) If $M$ is any $\pi$-lattice and $k$ is a field with $\zeta_n \in k$, then $k(M)^{\pi}$ is stably rational over $k$.

(ii) Suppose that $G:=A \rtimes D_n$ where $A$ is a finite abelian group of exponent $e$ and $n$ is an odd integer. If $k$ is a field with $\zeta_e, \zeta_n \in k$, then $k(G)$ is stably rational over $k$.
\end{lemma}

\begin{proof}
By Theorem \ref{t5.3}, $F_\pi=0$.

(i) Since $F_\pi=0$, we find that $[M]^{fl}=0$. The idea of showing that $k(M)^{\pi}$ is stably rational is almost the same as that of Lemma \ref{l4.12}. It remains to prove that $k(\pi^{\prime \prime})$ is rational  where $\pi^{\prime \prime}$ is some quotient group of $\pi$.

Since $n$ is odd, the quotients groups of $\pi \simeq D_n$ are isomorphic to $D_m$ ($m$ divides $n$), $C_2$, or the trivial group.

Suppose that $\pi^{\prime \prime}$ is a dihedral group. Write $\pi^{\prime \prime}= \langle \sigma_0, \tau:\sigma_0^{m}=\tau^2=1, \tau^{-1} \sigma_0 \tau = \sigma_0^{-1} \rangle$ where $m$ is some integer dividing $n$. Define $\pi_0 := \langle \tau \rangle \simeq C_2$. If char $k \neq 2$, apply Theorem \ref{t1.1}. Here is a proof for the general case.

Since $\zeta_n \in k$, we may use the same arguments as in Step 1 of the proof of Theorem \ref{t4.3}. We find that $k(\pi^{\prime \prime})$ is rational over $k(M)^{\pi_0}$ where $M$ is a projective ideal of $\bm{Z}\pi_0$. Since $\pi_0 \simeq C_2$, any projective ideal of $\bm{Z}\pi_0$ is isomorphic to the free module $\bm{Z}\pi_0$ by Reiner's Theorem \cite{Re}. It follows that $k(M)^{\pi_0} \simeq k(C_2) =k(x,y)^{\pi_0}$ where $\tau: x \mapsto y \mapsto x$. Define $X=y/x$. Then $k(x,y)^{\pi_0}=k(X,x)^{\pi_0}$ is rational over $k(X)^{\pi_0}$ by Theorem \ref{t2.6}. Note that $k(X)^{\pi_0}=k(X/(1+X^2))$ is rational. Done.

If $\pi^{\prime \prime} \simeq C_2$, the rationality of $k(C_2)$ has been proved in the above paragraph.

\medskip
(ii) By the same method as in the proof of Theorem \ref{t1.6} (given in Section $4$), it can be shown that $k(G)$ is rational over $k(M)^{D_n}$ where $M$ is some $D_n$-lattice. Then apply the result of Part (i).
\end{proof}

\begin{remark}
Let $\pi=D_n$, the dihedral group of order $2n$. It is important that we assume that $n$ is odd in Theorem \ref{t5.3} and Lemma \ref{l5.4}. If $4 \mid n$, by \cite{Ba}, there exist $\pi$-lattices $M$ such that the unramified Brauer groups of $\bm{C}(M)^{\pi}$ are not zero; thus $\bm{C}(M)^{\pi}$ are not retract rational and hence are not stably rational. For a concrete construction of such lattices, see \cite[Theorem 6.2]{HKY}. If $2 \mid n$ and $4 \nmid n$, by \cite{Ba} again, the unramified Brauer group of $\bm{C}(M)^{\pi}$ is zero for any $\pi$-lattice $M$, but it is unknown whether $\bm{C}(M)^{\pi}$ is always stably rational.
\end{remark}

\newpage

\end{document}